\documentclass[11pt]{amsart}
\usepackage{amssymb,latexsym,amsmath}

\usepackage{bm,color}
\usepackage[all]{xy}

\newcommand{\cK}{\mathcal{K}}     
\newcommand{\cL}{\mathcal{L}}
\newcommand{\cN}{{\mathcal{N}}}  
\newcommand{\cS}{{\mathcal{S}}}

\newcommand{\C}{\mathbb{C}}

\newcommand{\N}{\mathbb{N}}
\newcommand{\Ann}{\mathrm{Ann}}  

\newcommand{\cst}{C$^*$}

\newtheorem{thm}{Theorem}[section]
\newtheorem{cor}[thm]{Corollary}
\newtheorem{lemma}[thm]{Lemma}
\newtheorem{prop}[thm]{Proposition}

\theoremstyle{definition}
\newtheorem{definition}[thm]{Definition}

\newtheorem{ex}[thm]{Example}
\newtheorem{rem}[thm]{Remark}

\newcommand{\nai}[2]{\langle #1,#2\rangle}
\newcommand{\eqa}[1]{
\begin{align*}
#1
\end{align*}}
\newcommand{\mattwo}[4]{\begin{pmatrix}#1 & #2\\ #3 & #4\end{pmatrix}}
\numberwithin{equation}{section}
\begin{document}

\title{Non-commutativity of the central sequence algebra for separable non-type {\rm{I}} \cst-algebras}

\author{Hiroshi  Ando}
\address{Department of Mathematics and Informatics,\\
Chiba University\\
1-33 Yayoi-cho, Inage, Chiba,\\
263-8522 Japan}
\email{hiroando@math.s.chiba-u.ac.jp}

\author{Eberhard Kirchberg}
\address{Institut f{\"u}r Mathematik,
Humboldt Universit{\"a}t zu Berlin, 
Unter den Linden 6,\\
D--10099 Berlin, Germany}
\email{kirchbrg@mathematik.hu-berlin.de}

\date{Version: \today}

\begin{abstract}
We show that if $A$ is a (not necessarily unital) separable, simple and non-type I C$^{\ast}$- algebra, then for every properly infinite hyperfinite von Neumann algebra $M$ with separable predual, its Ocneanu central sequence algebra $M'\cap M^{\omega}$ arises as a sub-quotient of the central sequence algebra $F(A)$ defined by the second named author. In particular, this answers affirmatively the question of the second named author in \cite{Kir.AbelProc}: the central sequence C$^{\ast}$-algebra of the reduced free group C$^{\ast}$-algebra $C_{\rm{red}}^*(\mathbb{F}_2)$ is non-commutative. 
\end{abstract}

\maketitle

\section{Introduction and main Results}
\label{sec:intro}

Let $A$ be a \cst-algebra, and let $\omega$ be a free ultrafilter on $\mathbb{N}$. The central sequence algebra of $A$ is the relative commutant $A'\cap A_{\omega}$ of $A$ inside the norm-ultrapower $A_{\omega}$.  Recent developments in the classification of C$^{\ast}$-algebras show the importance of the analysis on $A'\cap A_{\omega}$. On the other hand, if $A$ is non-unital, $A'\cap A_{\omega}$ is often too large even for type I C$^{\ast}$-algebras such as the compact operator algebra $\mathcal{K}$. 
In \cite[sec.1]{Kir.AbelProc}
the second named author introduced the 
invariant $F(A):= (A'\cap  A_\omega)/  \mathrm{Ann}(A, A_\omega)$ which has better behavior for non-unital C$^{\ast}$-algebras than the usual central sequence algebras. In fact, it is shown that (for a fixed $\omega$) $F(A)$ is a stable invariant for $\sigma$-unital C$^{\ast}$-algebras, while the central sequence algebras are clearly not. 
In any case, the central sequences in C$^{\ast}$-algebras have rather different properties from   
those in von Neumann algebras. 

It is now  understood that properties of the invariant $F(A)$, and 
its
continuous analogs, are important for the study of separable amenable
\cst-algebras.

For example it is known that $A\otimes  \cK \cong \cK$
if $F(A)\cong \C$ and  $A$ is separable.
A separable C$^*$-algebra $A$ is nuclear, simple and purely infinite
if and only if $F(A)$ is simple and $F(A)\not\cong \C$, in which case $F(A)$ is also purely infinite.

These and other properties of $F(A)$ can be found in \cite{Kir.AbelProc}
and \cite{KirRor.Crelle2013}.

It is in particular  interesting to know when $F(A)$ has no character (see the recent work of the second named author and R\o rdam \cite{KirchbergRordam14}).  
A small step towards this direction is
the main result of this paper stated as the following theorem (in this paper, we do not assume that C$^*$-algebras are unital unless stated otherwise explicitly):
\begin{thm}\label{thm:main}
If $A$ is a separable {\rm{C}}$^*$-algebra that is not of type {\rm{I}},
then $F(A)$ is not commutative. 
\end{thm}
In particular,  the central sequence algebra of the 
reduced group C*-algebra $C^*_{{\rm red}}(\mathbb{F}_n)$
of the free group $\mathbb{F}_n$ ($n=2,3,\ldots $) is not commutative
despite the W$^*$-central sequence algebra 
of the group von Neumann algebra $L(\mathbb{F}_n)\subset \cL(\ell_2(\mathbb{F}_n))$ being trivial. This is an affirmative answer to a question by the second named author \cite[Question 2.16]{Kir.AbelProc}
 (by Akemann-Pedersen Theorem \cite{AkemannPedersen}, it was known before that the central sequence algebra is not isomorphic to $\mathbb{C}$).   
The above result is inspired by (and is in fact a generalization of) the following result by the second named author and R\o rdam \cite[Theorem 3.3]{KirRor.Crelle2013}, which in turn is a generalization of the work of Sato \cite[Lemma 2.1]{Sato11}:
\begin{thm}[Kirchberg-R\o rdam,\ Sato]\label{thm: SKR theorem} Let $A$ be a separable unital ${\rm C}^\ast$-algebra, with a faithful tracial state $\tau$. Let $N$ be the weak closure of $A$ under the GNS representation $\pi_{\tau}$ of $A$ with respect to $\tau$, and let $\omega$ be a free ultrafilter on $\mathbb{N}$. Then the natural *-homomorphisms
$$A_{\omega}\to N^{\omega},\ \ \ A_{\omega}\cap A'\to N^{\omega}\cap N'$$
 are surjective. 
\end{thm}
Sato \cite{Sato11} proved the above result for the case of nuclear C$^{\ast}$-algebras and the general case was proved in \cite{KirRor.Crelle2013}. Theorem \ref{thm: SKR theorem} shows that if $A$ has a faithful tracial state $\tau$ such that $\pi_{\tau}(A)''$ is non-McDuff, then $F(A)$ has a character. Our proof follows closely this idea of mapping the C$^*$-central sequence algebra onto the W$^*$-central sequence algebra of some GNS representation of the original algebra. However, in order to prove the non-commutativity of $F(A)$ for arbitrary non-type I C$^\ast$ algebras, we have to use non-tracial W$^*$-ultraproducts and its central sequence algebras. Also, we have to pass to sub-quotients instead of genuine quotients in order to estimate the size of $F(A)$: the second named author has shown that $F(C_{\rm{red}}^\ast (\mathbb{F}_2))$ is stably finite (see \cite[$\S$2]{Kir.AbelProc}), while W$^*$-central sequence algberas can be a type III factor. Thus we prove the following theorem, which implies Theorem \ref{thm:main}:
\begin{thm}\label{thm: maindesu}Let $A$ be a non-type ${\rm{I}}$ separable ${\rm{C}}^{\ast}$-algebra. Then for every properly infinite hyperfinite von Neumann algebra $M$ with separable predual, there exists a closed ideal $I$ of $F(A)$, a ${\rm C}^{\ast}$-subalgebra $B$ of $F(A)/I$ and a closed ideal $J$ of $B$ such that $B/J$ is isomorphic to $M'\cap M^{\omega}$, where $M^{\omega}$ is the Ocneanu ultrapower of $M$.   
\end{thm} 
If $A$ is moreover simple, then we may choose $I=\{0\}$, so that $F(A)$ contains uncountably many non-separable (in the W$^*$- sense) type III factors (e.g. the Ocneanu central sequence algebras of Powers factors) as sub-quotients. 
\begin{rem}[Added November 4, 2014] 
After a seminar talk in Kyoto University, the first named author was informed from Professor Narutaka Ozawa that the non-commutativity of $F(A)$ for a separable non-type I C$^*$-algebra follows from the work of Kishimoto-Ozawa-Sakai \cite{KOS03}. We include his proof in the Appendix. 
\end{rem}

\medskip

\section{Preliminaries and notations}\label{sec:prelim.notat}
\subsection{Central sequences in C$^*$-algebras and the invariant $F(A)$}
Throughout the paper, fix a free ultrafilter $\omega$ on $\N$. For a sequence $\bm{A}=(A_1,A_2,\dots)$ of C$^*$-algebras, we denote by $\ell_{\infty}(\bm{A})$ the C$^{\ast}$-algebra of all bounded sequences $(a_1,a_2,\dots)\in \prod_{n\in \mathbb{N}}A_n$. In this section we only consider the constant algebra case $A_n\equiv A$, and we denote $\ell_{\infty}(\bm{A})$ as $\ell_{\infty}(A)$.

The C$^*$-ultraproduct algebra $A_\omega$ is defined by $A_\omega :=\ell_\infty(A)/c_\omega(A)$,
where the closed ideal $c_\omega(A)$ of $\ell_{\infty}(A)$ consists of the sequences 

$(a_1,a_2,\ldots)\in \ell_\infty(A)$
with $\lim_{n\to \omega}  \| a_n\|=0$.

We denote the quotient epimorphism   
$\ell_{\infty}(A)\ni (a_1,a_2,\ldots)\mapsto (a_1,a_2,\ldots)+c_\omega(A)\in A_{\omega}$
by $\pi_\omega$.
We define ultra-powers $T_\omega\colon B_\omega \to A_\omega$  of 
a bounded linear map $T\colon B\to A$ in a similar way 
by 
$$T_\omega ((b_1,b_2,\ldots)+c_\omega(B)):= 
(T(b_1),T(b_2),\ldots)+c_\omega (A)\,.$$

The elements $a\in A$ will be naturally identified with 
the image $\pi_\omega(\Delta(a))\in A_\omega$
of  $\Delta(a):=(a,a, \ldots)$. 

Then  $A'\cap A_\omega$ is the natural image in $A_\omega$ of the bounded
\emph{$\omega$-central sequences}   $(a_1,a_2,\ldots)\in \ell_\infty(A)$  defined 
by
$$\lim_{n\to \omega} \| a_nb -ba_n \| =0, \quad \ \ b\in A\,.$$
The (two-sided) {\it annihilator} $\Ann (A, A_\omega)$ is a closed ideal of 
$A'\cap A_\omega$ defined as  the $\pi_\omega$-image of the $\omega$-approximately annihilating 
sequences
$(a_1,a_2,\ldots)\in \ell_\infty(A)$  that are  defined by
$$\lim_{n\to \omega}\| a_nb\| + \|ba_n \| =0,\quad  \,\, b\in A\,.$$
Finally we can form the quotient \cst-algebra $F(A)$ of this sequence algebras.
\begin{definition} Let $A$ be a C$^*$-algebra. The invariant $F(A)$ is defined as the quotient C$^*$-algebra 
$$F(A):=(A'\cap A_\omega)/ \Ann (A, A_\omega).$$  
\end{definition}
Note that $F(A)$ may depend on $\omega$ (see \cite{FarahPhillipsSteprans,Farah09}), but the structure of separable C$^*$-subalgebras of $F(A)$ do not. Therefore we use the notation $F(A)$.  
In the class of $\sigma$-unital  \cst-algebras $A$ is $F(A)$ a stable invariant.
See \cite[sec.~1]{Kir.AbelProc}. 
We do not have a general rule that shows  which  AF C*-algebras $A$ of type I
and for which sub-homogeneous C*-algebras $A$ have \emph{commutative} 
 invariant $F(A)$. We remark that there is a C$^*$-algebra $B$ with commutative $F(B)$ which contains a C$^*$-subalgebra $A$ with non-commutative $F(A)$.
\begin{ex}\label{ex: F(A) is NC} Consider the compact Hausdorff space $T=\{0\}\cup \{\frac{1}{n};n\in \mathbb{N}\}$. 
Define $B:= \mathrm{C}(T, M_2(\mathbb{C}))=M_2(C(T)))$. 
By direct calculation, we see that
\[F(B)=\left \{\mattwo{f}{0}{0}{f};f\in F(C(T))\right \}\cong F(C(T)),\]
which is commutative. 
On the other hand, consider its \cst-subalgebra 
$A:= \{ f\in B\,;\,\, f(0)\in \mathbb{C}1\}$. Then $F(A)$ is non-commutative (and is non-separable): to see this, for each $\alpha>0$ and $n\in \mathbb{N}$, define $f^{(n)}_{\alpha}\in A$ by $f_{\alpha}^{(n)}(0)=1$ and 
set $f_{\alpha}^{(n)}(\tfrac{1}{k})$ to be $\mattwo{1}{1}{0}{1}$ for $n\le k\le (\alpha+1)n$ and  $\mattwo{1}{0}{0}{1}$ otherwise.  
Then $f_{\alpha}:=(f^{(1)}_{\alpha},f^{(2)}_{\alpha},\dots)+c_{\omega}(A)\in A'\cap A_{\omega}$. Indeed, let $g=[g_{ij}]\in A$, where $g_{ij}\in C(T)$ with $g_{ij}(0)=\delta_{ij}\ (i,j=1,2)$ and let $\varepsilon>0$. 
Then since $\|g(\frac{1}{k})-1\|\stackrel{k\to \infty}{\to}0$, there exists $N\in \mathbb{N}$ such that the norm of 
\[X_k=\left [\mattwo{1}{1}{0}{1},\mattwo{g_{11}(\frac{1}{k})}{g_{12}(\frac{1}{k})}{g_{21}(\frac{1}{k})}{g_{22}(\frac{1}{k})}\right ]=\mattwo{g_{21}(\frac{1}{k})}{g_{22}(\frac{1}{k})-g_{11}(\frac{1}{k})}{0}{-g_{21}(\frac{1}{k})}\]
is less than $\varepsilon$ for $k\ge N$. We have 
\eqa{
(f_{\alpha}^{(n)}g-gf_{\alpha}^{(n)})(\tfrac{1}{k})=\begin{cases}
\ \ X_k & (n\le k\le (\alpha+1)n)\\
\ \ \ 0 & (\text{otherwise}).
\end{cases}
}

Then for $n\ge N$, one has $\sup_{k\in \mathbb{N}}\|[f_{\alpha}^{(n)}(\frac{1}{k}),g(\frac{1}{k})]\|<\varepsilon$. Since $\varepsilon>0$ is arbitrary, we see that 
$\lim_{n\to \omega}\|f_{\alpha}^{(n)}g-gf_{\alpha}^{(n)}\|=0$ for every $g\in A$. Therefore $f_{\alpha}\in A'\cap A_{\omega}\ (\alpha>0)$. 
Now let $\alpha>\beta>0$. There exists $N\in \mathbb{N}$ such that $(\alpha-\beta)N\ge 2$. Then for every $n\ge N$, there exists $k_n\in \mathbb{N}$ with $(\beta+1)n<k_n\le (\alpha+1)n$, so that 
\[\left \|f_{\alpha}^{(n)}(\tfrac{1}{k_n})-f_{\beta}^{(n)}(\tfrac{1}{k_n})\right \|=\left \|\mattwo{0}{1}{0}{0}\right \|=1,\]
whence $\|f_{\alpha}^{(n)}-f_{\beta}^{(n)}\|=1\ (n\ge N)$, and $\|f_{\alpha}-f_{\beta}\|=1$. Thus $F(A)$ is non-separable.  
Now let $f:=f_{1}, g:=f_{1}^*\in F(A)$. Then $fg\neq gf$, because for each $n$, one has 
\[[f^{(n)},g^{(n)}](\tfrac{1}{n})=\left [\mattwo{1}{1}{0}{1},\mattwo{1}{0}{1}{1}\right ]=\mattwo{1}{0}{0}{-1}.\]
Therefore $F(A)$ is non-commutative.  
\end{ex}

\section{The Ocneanu ultraproduct and C$^*$-to-W$^*$ ultraproduct}\label{sec: Ocneanu UP}
 We recall the definition of (generalized) Ocneanu ultraproduct \cite{Ocneanu85} of W$^*$-algebras. 
\begin{definition} 
Let $\bm{M}=(M_1,M_2,\dots)$ be a sequence of $\sigma$-finite W$^*$-algebras, and let $\bm{\rho}=(\rho_1,\rho_2,\dots)$ be a sequence of faithful normal states with $\rho_n\in (M_n)_*\ (n\in \mathbb{N})$. 
We define $D_{\bm{\rho}}$ to be the hereditary C$^*$-subalgebra of $\ell_{\infty}(M_1,M_2,\dots)$ consisting of those $(x_1,x_2,\dots)\in \ell_{\infty}(M_1,M_2,\dots)$
satisfying 
\[\lim_{n\to \omega}\|x_n\|_{\rho_n}^{\sharp}=0.\] 
Here, we used the standard notation $\|a\|_{\rho}=\rho(a^*a)^{\frac{1}{2}},\ \|a\|_{\rho}^{\sharp}=\rho(a^*a+aa^*)^{\frac{1}{2}}$. 
The normalizer algebra $\mathcal{N}(D_{\bm{\rho}})$ is then defined by $\{x\in \ell_{\infty}(\bm{M}); xD_{\bm{\rho}}+D_{\bm{\rho}}x\subset D_{\bm{\rho}}\}$. 
The {\it  Ocneanu ultraproduct} $(M_n,\rho_n)^{\omega}$ is defined as the quotient C$^*$-algebra $\mathcal{N}(D_{\bm{\rho}})/D_{\bm{\rho}}$ which is in fact a W$^*$-algebra.   
\end{definition}

\begin{rem}Ocneanu studied the constant algebra and constant state case $M_n\equiv M,\rho_n\equiv \rho$, in which case $(M,\rho)^{\omega}$ does not depend on the choice of $\rho$. Therefore we write $(M,\rho)^{\omega}$ as $M^{\omega}$ in this case. The non-constant algebra/state case is studied in \cite{AndoHaagerup}.   
\end{rem}

In a more general context, the second named author introduced the C$^*$-to-W$^*$-ultraproduct $(A_n,\rho_n)_{\omega}$ \cite{Kirchberg95} for a sequence $\bm{A}=(A_1,A_2,\dots)$ of C$^*$-algebras and states $\bm{\rho}=(\rho_1,\rho_2,\dots)$. 
Recall that an operator system $X$ is called a (unital) {\it {\rm{C}}$^*$-system}, if the second conjugate operator system $X^{**}$ is unital completely isometrically isomorphic (u.c.i.i) to a C$^*$-algebra, and every unital complete isometry (u.c.i) $V$ from $X^{**}$ onto a C$^*$-algebra $A$ induces on $X^{**}$ the structure of a C$^*$-algebra such that the given matrix order unit structure and the matrix order unit structure of the C$^*$-algebra coincide.   
\begin{definition}[C$^*$-to-W$^*$ ultraproduct]
Let $(\bm{A},\bm{\rho})=(A_n,\rho_n)_{n=1}^{\infty}$ be a sequence of C$^*$-algebras equipped with (not necessarily faithful) states. The {\it {\rm{C}}$^*$-to-{\rm{W}}$^*$-ultraproduct} $(A_n,\rho_n)_{\omega}$ is defined as the quotient C$^*$-system $\ell_{\infty}(\bm{A})/(L_{\bm{\rho}}+L_{\bm{\rho}}^*)$, where $L_{\bm{\rho}}$ is the closed left ideal of $\ell_{\infty}(\bm{A})$ consisting of those $(a_1,a_2,\dots)\in \ell_{\infty}(\bm{A})$ satisfying
$$\lim_{n\to \omega}\rho_n(a_n^*a_n)=0\,.$$
\end{definition}
It was shown in \cite{Kirchberg94} (with results from \cite{Kirchberg95}) that $(A_n,\rho_n)_{\omega}$ is c.i.i. to a W$^*$-algebra. For later discussions, let us include a brief summary. 
We see in fact that the Ocneanu ultraproduct can be identified with the special case of C$^*$-to-W$^*$ ultraproducts. 
Define a state $\rho_{\omega}$ on $\ell_{\infty}(\bm{A})$ by 
\[\rho_{\omega}((a_1,a_2,\dots)):=\lim_{n\to \omega}\rho_n(a_n^*a_n),\ \ \ \ (a_1,a_2,\dots)\in \ell_{\infty}(\bm{A}).\]
For each $n\in \mathbb{N}$, let $\overline{\rho}_n\in A_{n}^{**}$ be the extension of $\rho_n$ to a normal state on $A_{n}^{**}$, and let $p_n:=\text{supp}(\overline{\rho}_n)\in A_{n}^{**}$ be the support projection of $\overline{\rho}_n$. Define a von Neumann algebra $M_n:=p_nA_{n}^{**}p_n$ and a faithful normal state $\mu_n:=\overline{\rho}_n|_{M_n}$ on $M_n$. Let $\overline{\rho}_{\omega}$ be the extension of $\rho_{\omega}$ to a normal state on $\ell_{\infty}(\bm{A})^{**}$. 
Also, define a state $\mu_{\omega}$ on $\ell_{\infty}(\bm{M})$ as the pointwise $\omega$-limit of $(\mu_1,\mu_2,\dots)$.  Now define the closed left ideal $L_{\bm{\rho}}\subset \ell_{\infty}(\bm{A})$ (resp. $L_{\bm{\mu}}\subset \ell_{\infty}(\bm{M})$) with respect to $\bm{A}=(A_1,A_2,\dots), \bm{\rho}=(\rho_1,\rho_2,\dots)$ (resp. $\bm{M}=(M_1,M_2,\dots), \bm{\mu}=(\mu_1,\mu_2,\dots)$) as before.  
Then by \cite[Proposition 2.1 (iii)]{Kirchberg94}, we have 
\begin{equation}
\ell_{\infty}(\bm{A})/(L_{\bm{\rho}}+L_{\bm{\rho}}^*)\stackrel{\rm{c.i.i.}}{\cong}\ell_{\infty}(\bm{M})/(L_{\bm{\mu}}+L_{\bm{\mu}}^*)\,.\label{eq: we may reduce to W* case}
\end{equation}
Therefore in order to show that $(A_n,\rho_n)_{\omega}=\ell_{\infty}(\bm{A})/(L_{\bm{\rho}}+L_{\bm{\rho}}^*)$ is c.i.i. to a W$^*$-algebra, we may assume that each $A_n$ is a W$^*$-algebra and $\rho_n$ is a normal faithful state. Moreover, it follows that (\cite[Proposition 2.1, Lemma 2.3]{Kirchberg94})
\begin{equation}
\ell_{\infty}(\bm{A})=L_{\bm{\rho}}+L_{\bm{\rho}}^*+\mathcal{N}(D_{\bm{\rho}}).\label{eq: Kirchberg decomposition}
\end{equation}
An alternative proof of (\ref{eq: Kirchberg decomposition}) for the case of W$^*$-algebras/normal faithful states is given in \cite[Proposition 3.14]{AndoHaagerup}. Now, since $\mathcal{N}(D_{\bm{\mu}})\cap (L_{\bm{\mu}}+L_{\bm{\mu}}^*)=L_{\bm{\mu}}\cap L_{\bm{\mu}}^*$ (see. e.g. \cite[Lemma 3.10 (2)]{AndoHaagerup}), we have (by (\ref{eq: we may reduce to W* case})) 
$$(A_n,\rho_n)_{\omega}\stackrel{\rm{c.i.i.}}{\cong} \mathcal{N}(D_{\bm{\mu}})/(L_{\bm{\mu}}+L_{\bm{\mu}}^*)=\mathcal{N}(D_{\bm{\mu}})/D_{\bm{\mu}}=(M_n,\mu_n)^{\omega}\,.$$
That is, the C$^*$-to-W$^*$ ultraproduct $(A_n,\rho_n)_{\omega}$ is naturally identified with the generalized Ocneanu ultraproduct $(M_n,\mu_n)^{\omega}$. 
In the sequel, we identify $(A_n,\rho_n)_{\omega}=\mathcal{N}(D_{\bm{\rho}})/D_{\bm{\rho}}$. 
If each $A_n$ is a weakly dense C$^*$-subalgebra of a W$^*$ algebra, then we do not need to pass to $A_n^{**}$. 
\begin{prop}\label{prop: connection between C* to W* and Ocneanu}
For each $n\in \mathbb{N}$, let $\mu_n$ be a normal faithful state on a {\rm{W}}$^*$-algebra $M_n$, and let $A_n\subset N_n$ be a weakly dense {\rm{C}}$^*$-subalgebra. Set $\rho_n:=\mu_n|_{A_n}$. Then the {\rm{C}}$^*$-to-{\rm{W}}$^*$ ultraproducts $(A_n,\rho_n)_{\omega}$ and $(M_n,\mu_n)_{\omega}$ are naturally isomorphic.
\end{prop}

\begin{proof}
Let $L_{\bm{\mu}}\subset \ell_\infty(\bm{M})$ (resp. $L_{\bm{\rho}}\subset \ell_{\infty}(\bm{A})$)
denote the closed left-ideal defined by $\bm{\mu}=(\mu_1,\mu_2,\dots)$ (resp. $\bm{\rho}=(\rho_1,\rho_2,\dots)$). 

Then it holds that $L_{\bm{\rho}}= L_{\bm{\mu}}\cap \ell_\infty(\bm{A})$, and $D_{\bm{\rho}}=L_{\bm{\rho}}\cap L_{\bm{\rho}}^*=D_{\bm{\mu}}\cap \ell_{\infty}(\bm{A})$. 
We show that $\mathcal{N}(D_{\bm{\rho}})=\mathcal{N}(D_{\bm{\mu}})\cap \ell_{\infty}(\bm{A})$. It is clear that $\mathcal{N}(D_{\bm{\rho}})\supset \mathcal{N}(D_{\bm{\mu}})\cap \ell_{\infty}(\bm{A})$. Conversely, let $x=(x_1,x_2,\dots)\in \mathcal{N}(D_{\bm{\rho}})$. We must show that $x\in \mathcal{N}(D_{\bm{\mu}})$. Given $e=(e_1,e_2,\dots)\in D_{\bm{\mu}}$. 
By Kaplansky density Theorem, for each $n\in \mathbb{N}$, there exists $\tilde{e}_n\in A$ with $\|\tilde{e}_n\|\le \|e_n\|$ such that 
$\|(e_n-\tilde{e}_n)x_n\|_{\mu_n}^{\sharp}+\|x_n(e_n-\tilde{e}_n)\|_{\mu_n}^{\sharp}<\frac{1}{n}$ and 
$\|e_n-\tilde{e}_n\|_{\mu_n}^{\sharp}<\frac{1}{n}$.  In particular, 
$\tilde{e}=(\tilde{e}_1,\tilde{e}_2,\dots)\in D_{\bm{\rho}}$, and 
\[\lim_{n\to \omega}(\|e_nx_n\|_{\mu_n}^{\sharp}+\|x_ne_n\|_{\mu_n}^{\sharp})=\lim_{n\to \omega}(\|\tilde{e}_nx_n\|_{\rho_n}^{\sharp}+\|x_n\tilde{e}_n\|_{\rho_n}^{\sharp})=0,\]
by $x\in \mathcal{N}(D_{\rm{\rho}})$. 
Therefore the natural inclusion $\ell_{\infty}(\bm{A})\hookrightarrow \ell_{\infty}(\bm{M})$ induces a natural injective *-homomorphism
\[\Psi\colon (A_n,\rho_n)_{\omega}=\frac{\mathcal{N}(D_{\bm{\rho}})}{D_{\bm{\rho}}}=\frac{\mathcal{N}(D_{\bm{\mu}})\cap \ell_{\infty}(\bm{A})}{D_{\bm{\mu}}\cap \ell_{\infty}(\bm{A})}\to \frac{\mathcal{N}(D_{\bm{\mu}})}{D_{\bm{\mu}}}=(M_n,\mu_n)_{\omega}.\]
Also, $\Psi$ is surjective: if $x=(x_1,x_2,\dots)\in \mathcal{N}(D_{\bm{\mu}})$, then again by Kaplansky density Theorem, for each $n\in \mathbb{N}$ there exists $y_n\in A$ with $\|y_n\|\le \|x_n\|$ such that $\|x_n-y_n\|_{\mu_n}^{\sharp}<\frac{1}{n}$. Then $e:=(x_1-y_1,x_2-y_2,\dots)\in D_{\bm{\mu}}$, $y:=(y_1,y_2,\dots)\in \mathcal{N}(D_{\bm{\mu}})\cap \ell_{\infty}(\bm{A})=\mathcal{N}(D_{\bm{\rho}})$ and 
$\Psi(y+D_{\bm{\rho}})=x+D_{\bm{\mu}}$. Therefore $\Psi$ is a *-isomorphism.  
\end{proof}
\section{Proof of the Main Theorem}
\subsection{Reduction to W$^*$-algebra Case}
Let $A$ be a separable non-type I C$^{\ast}$-algebra. In this section, we show that in order to prove the non-commutativity of $F(A)$, we may assume that $A$ sits inside a hyperfinite properly infinite von Neumann algebra $M$. This follows from works of Glimm \cite{Glimm61}, Mar\'echal \cite{Marechal75} and Elliott-Woods \cite{ElliottWoods76}. Then in the next section we show that $F(A)$ contains $M'\cap M^{\omega}$ as a sub-quotient. For our purpose, the notion of $\sigma$-ideals introduced by the second named author plays a key role. Let us recall its definition and important consequences. 
\begin{definition}[\cite{Kir.AbelProc}]\label{def: sigma ideal}Let $I$ be a closed ideal of a C$^*$-algebra $A$. We call $I$ a {\it $\sigma$-ideal} of $A$, if for every separable C$^*$-subalgebra $B\subset A$ and every $d\in I_+$, there is a positive contraction $e\in B'\cap I$ with $ed=d$.
\end{definition}

\begin{thm}\cite[Proposition 1.6]{Kir.AbelProc}\label{thm: pi_I is surjective}
Let $I$ be a $\sigma$-ideal of a {\rm{C}}$^*$-algebra $A$. 
Then for every separable {\rm{C}}$^*$-subalgebra $C\subset A$, the sequence
\begin{equation}
0\to C'\cap I\to C'\cap A\stackrel{\pi_I}{\to} \pi_I(C)'\cap (A/I)\to 0\label{eq: KAP1.6}
\end{equation}
is exact ($\pi_I$ is the quotient map).  
\end{thm}

\begin{rem}
It is proved in \cite[Proposition 1.6]{Kir.AbelProc} that the sequence (\ref{eq: KAP1.6}) is not only exact but also {\it strongly locally semi-split}. That is, for every separable C$^*$-subalgebra $B\subset \pi_I(C)'\cap (A/I)$, there is a *-homomorphism $\psi\colon C_0((0,1])\otimes B\to C'\cap A$ such that $\pi_I\circ \psi(\iota\otimes b)=b\ (b\in B)$, where $\iota(t)=t,\ t\in (0,1]$. 
\end{rem}
\begin{prop}\cite[Corollary 1.7]{Kir.AbelProc}\label{prop: J_omega is sigmaideal}\ Let $A$ be a {\rm{C}}$^*$-algebra and $J$ be a norm-closed ideal of $A$. Then $J_{\omega}$ is a $\sigma$-ideal of $A_{\omega}$. 
\end{prop}

\begin{cor}\cite[Remark 1.15(3)]{Kir.AbelProc}\label{cor: F(A/J) is a quotient of F(A)}
Let $A$ be a separable {\rm{C}}$^*$-algebra, and let $J\triangleleft A$ be a closed ideal of $A$. Then $F(A/J)$ is a quotient of $F(A)$. 
\end{cor}
\begin{proof}
We include the proof for the reader's convenience. 
Let $\pi_J\colon A\to A/J$ be the quotient map, and $(\pi_J)_{\omega}\colon A_{\omega}\to (A/J)_{\omega}$ be its ultrapower map. 
It is straightforward to see that $(\pi_J)_{\omega}$ is surjective with kernel $J_{\omega}$. 
By the definition of $\text{Ann}(A,A_{\omega})$, it holds that
\begin{equation}
(\pi_J)_\omega ({\rm{Ann}} (A, A_\omega) ) \subseteq {\rm{Ann}} (A/J, (A/J)_\omega).\label{eq: Ann(A,A_{omega}) is contained in the Ann of A/J}
\end{equation}
and $(\pi_J)_\omega (A'\cap A_\omega) \subseteq  (A/J )' \cap (A/J)_\omega\,.$
Moreover, by Proposition \ref{prop: J_omega is sigmaideal}, $J_{\omega}$ is a $\sigma$-ideal. Therefore by Theorem \ref{thm: pi_I is surjective}, $(\pi_J)_{\omega}|_{A'\cap A_{\omega}}\colon A'\cap A_{\omega}\to (A/J)'\cap (A/J)_{\omega}$ is surjective . From this and  (\ref{eq: Ann(A,A_{omega}) is contained in the Ann of A/J}), we see that 
$F(A/J)=(A/J)'\cap (A/J)_{\omega}/{{\rm{Ann}}(A/J,(A/J)_{\omega})}$ is a quotient of $(A/J)'\cap (A/J)_{\omega}/(\pi_J)_{\omega}({\rm{Ann}}(A,A_{\omega}))$, which is a quotient of $F(A)=A'\cap A_{\omega}/\text{Ann}(A,A_{\omega})$. Therefore the claim follows.
\end{proof}

Next, recall that a combination of the results of Glimm \cite{Glimm61}, Mar\'echal \cite{Marechal75} and Elliott-Woods \cite{ElliottWoods76} yields the following theorem. 
\begin{thm}[Glimm, Mar\'echal, Elliott-Woods]\label{prop: generalization of Elliott-Woods}
Let $A$ be a separable non-type {\rm{I}} {\rm{C}}$^*$-algebra, and let $M$ be an injective properly infinite von Neumann algebra with separable predual. Then there exists a *-representation $d\colon A\to M$ such that $d(A)$ is weakly (hence ultra *-strongly) dense in $M$.   
\end{thm}
\begin{proof}
Glimm \cite{Glimm61} has shown that if $A$ is a separable non-type I C$^*$-algebra, then for each Powers factor $R_{\lambda}\ (0<\lambda<1)$ there exists a *-homomorphism $\pi\colon A\to R_{\lambda}$ with $\pi(A)''=R_{\lambda}$. 

Based on Glimm's work, Mar\'echal \cite[Proposition 2]{Marechal75} has extended this result to the following: 
let $A$ be a separable, non-type I ${\rm{C}}^*$-algebra and let $M$ be a properly infinite von Neumann algebra acting on a separable Hilbert space $H$ for which there exists a *-homomorphism $\pi\colon M_{2^{\infty}}\to \cL(H)$ satisfying $\pi(M_{2^{\infty}})''=M$. Then there exists a *-homomorphism $\rho\colon A\to \cL(H)$ such that $\rho(A)''=M$.

By Elliott-Woods Theorem \cite{ElliottWoods76}, any properly infinite hyperfinite von Neumann algebra $M$ with separable predual contains a weakly dense copy of the CAR algebra $M_{2^{\infty}}$. The combination of these results finishes the proof.
\end{proof}
Now we can reduce the proof of Theorem \ref{thm:main} to the following stronger result: 
\begin{thm}\label{thm: main2}
Let $M$ be a von Neumann algebra with separable predual, and let $A$ be a (not necessarily unital) separable {\rm{C}}$^*$-subalgebra of $M$ which is weakly dense in $M$. Then there exists a {\rm{C}}$^*$-subalgebra $B$ of $F(A)$, and a closed ideal $J$ of $B$ such that $B/J\cong M'\cap M^{\omega}$, where $M^{\omega}$ is the Ocneanu (or equivalently, {\rm{C}}$^*$-to-{\rm{W}}$^*$) ultrapower of $M$. 
\end{thm}
The proof of the above theorem will be given in the next section. 
Now the main theorem is proved as follows: 
\begin{proof}[Proof of Theorem \ref{thm:main}]
Let $M$ be a properly infinite injective von Neumann algebra with separable predual. 
By Theorem \ref{prop: generalization of Elliott-Woods}, there exists a *-representation $d\colon A\to M$ such that $d(A)''=M$.  
Then apply Theorem \ref{thm: main2} to $d(A)\subset M$ to get that $F(d(A))$ contains isomorphic copies of the Ocneanu central sequence algebras $M'\cap M^{\omega}$ as a sub-quotient. If in particular we choose $M$ to be the Powers factor $R_{\lambda}$ of type III$_{\lambda}\ (0<\lambda<1)$, we get that $F(d(A))$ contains a type III$_{\lambda}$ factor $R_{\lambda}'\cap R_{\lambda}^{\omega}$ (see \cite[Example 5.1]{AndoHaagerup}) with non-separable predual as a sub-quotient. By Corollary \ref{cor: F(A/J) is a quotient of F(A)}, $F(d(A))$ is a quotient of $F(A)$. This shows that $F(A)$ is non-commutative and non-separable.
\end{proof}
\begin{cor}
Let $A$ be a unital simple separable {\rm{C}}$^*$-algebra that is not of type {\rm{I}}. Then for each injective type {\rm{III}} factor $M$ with separable predual, the Ocneanu central sequence algebra $M'\cap M^{\omega}$arises as a sub-quotient of $A'\cap A_{\omega}$. 
\end{cor}  
\begin{rem}
Since $M'\cap M^{\omega}$ is not of type III if $M$ is a (hyperfinite) type III$_0$ factor (see \cite[Theorem 6.18]{AndoHaagerup}), we do not know whether a type III$_0$ factor arises as a sub-quotient of $F(A)$. 
\end{rem}
\subsection{$\sigma$-ideals and embedding of $\Delta(A)$ into the normalizer of $D_{\rho}$}
We continue to keep the notation from $\S$\ref{sec: Ocneanu UP}. 
Thus we let $\bm{A}=(A_1,A_2,\dots)$ and $\bm{\rho}=(\rho_1,\rho_2,\dots)$ be a sequence of C$^*$-algebras and states. We define $L_{\bm{\rho}}$, $D_{\bm{\rho}}=L_{\bm{\rho}}\cap L_{\bm{\rho}}^*$  and $\mathcal{N}(D_{\bm{\rho}})$ as before and $\pi_{\omega}$ denotes the quotient map $\ell_{\infty}(\bm{A})\to \ell_{\infty}(\bm{A})/c_{\omega}(\bm{A})$. 
Our strategy is to find an appropriate $\sigma$-ideal which would allow us to map a certain central sequence-like subalgebra of $F(A)$ onto the W$^*$-central sequence algebra. A natural candidate might be $D_{\bm{\rho}}\triangleleft \mathcal{N}(D_{\bm{\rho}})$. However, it is not clear whether $D_{\bm{\rho}}$ is actually a $\sigma$-ideal of $\cN (D_{\bm{\rho}})$. However, the problem can be resolved by passing to the quotient by $c_{\omega}(\bm{A})$: 
\begin{prop}\label{prop: Obs (7)}
$\pi_\omega(D_{\bm{\rho}})$ is a $\sigma$-ideal of $\pi_\omega(\cN (D_{\bm{\rho}}))$.
\end{prop}

For the proof, we use the next lemma (see \cite[Lemma A.1]{Kir.AbelProc} or \cite[Lemma 3.1]{KirRor.Crelle2013} for the proof):
\begin{lemma}[The $\varepsilon$-test]\label{lem: epsilon test}
Let $\omega$ be a free ultrafilter on $\mathbb{N}$. Let $X_1,X_2,\dots$ be any sequence of sets. Suppose that for each $k\in \mathbb{N}$, we are given a sequence $(f_n^{(k)})_{n=1}^{\infty}$ of functions $f_n^{(k)}\colon X_n\to [0,\infty)$. For each $k\in \mathbb{N}$, define a new function $f_{\omega}^{(k)}\colon \prod_{n=1}^{\infty}X_n\to [0,\infty]$ by 
\[f_{\omega}^{(k)}(s_1,s_2,\dots)=\lim_{n\to \omega}f_n^{(k)}(s_n),\ \ \ (s_n)_{n=1}^{\infty}\in \prod_{n=1}^{\infty}X_n.\]
Suppose that for each $m\in \mathbb{N}$ and each $\varepsilon>0$, there exists a sequence $s=(s_1,s_2,\dots)\in \prod_{n=1}^{\infty}X_n$ such that 
\[f_{\omega}^{(k)}(s)<\varepsilon\text{\ \ \ \ for\ }k=1,2,\dots,m.\]
Then there exists a sequence $t=(t_1,t_2,\dots) \in \prod_{n=1}^{\infty}X_n$ with 
\[f_{\omega}^{(k)}(t)=0,\ \ \ \text{for all\ }k\in \mathbb{N}.\]
\end{lemma}

\begin{proof}[Proof of Proposition \ref{prop: Obs (7)}]
Let $B\subset \pi_{\omega}(\mathcal{N}(D_{\bm{\rho}}))$ be a separable ${\rm C}^\ast$-subalgebra. 
Then there is a countable subset $S=\{s_n\}_{n=1}^{\infty}$ of $\mathcal{N}(D_{\bm{\rho}})$, where 
$s_n=(s_1^{(n)},s_2^{(n)},\dots)\in \ell_{\infty}(\bm{A})$, such that $\pi_{\omega}(S)$ is dense in $B$. Let $d$ be a positive contraction in $\pi_{\omega}(D_{\bm{\rho}})$, and let $y=(y_1,y_2,\dots)\in D_{\bm{\rho}}$ be a sequence of positive contractions satisfying $\pi_{\omega}(y)=d$.  

Using Lemma \ref{lem: epsilon test}, we are going to construct a sequence $e=(e_1,e_2,\ldots) \in D_{\bm{\rho}}$ of positive contractions  with $y-ey\in c_\omega(\bm{A})$  and
$s_ne-es_n\in c_\omega(\bm{A})$ for all  $n\in \N$. Then $\tilde{e}=\pi_{\omega}(e)\in \pi_{\omega}(D_{\bm{\rho}})$ would be the required positive contraction in Definition \ref{def: sigma ideal} of a $\sigma$-ideal.  

We define sets $X_n$ and functions $f_n^{(k)}\colon X_n\to [0,\infty)$ by 
$X_n:= (A_n)_+^{\leq 1}$, $f_n^{(1)}(x_n):= \rho(x_n^*x_n+x_nx_n^*), f_n^{(2)}(x_n):=\| y_n - x_ny_n\|$
and $f_n^{(k+2)}(x_n):= \| [x_n, s^{(k)}_n] \|$  for
$k,n \in \N$. 

For $(x_1,x_2,\ldots)$ with $x_n\in X_n$,  we define
$$f^{(k)}_\omega(x_1,x_2, \ldots):= \lim_{n\to \omega} f_n^{(k)}(x_n).$$

Consider the separable \cst-algebra $C:= C^*(S\cup \{ y\})\subset \cN (D_{\bm{\rho}})$ 
and let $I:=  C\cap D_{\bm{\rho}}$.  $I$ is a closed ideal of the separable C$^{\ast}$-algebra $C$ that contains $y$.

Let $\{e^{(p)}=(e_1^{(p)},e_2^{(p)},\dots)\}_{p=1}^{\infty}$ be an approximate unit of $I$ consisting of positive contractions which is quasi-central for $C$. 

Then for given $m\in \N$ and $\varepsilon>0$, 
we find $p\in \N$ such that 
$\| [e^{(p)}, s_n] \| <\varepsilon$ for $1\le n\leq m$
and  $\| y -e^{(p)} y\|<\varepsilon$. Since also $e^{(p)}\in D_{\bm{\rho}}$,
we get that the sequence $(x_1,x_2, \ldots), x_n=e^{(p)}_n$ satisfies the $\varepsilon$-test  $f^{(k)}_\omega(x_1,x_2,\ldots)<\varepsilon,\ \ k=1,\dots, m+2$. Lemma \ref{lem: epsilon test} then finishes the proof. 
\end{proof}
From now on we only consider the constant case $A_n\equiv A, \rho_n\equiv \rho$.
\begin{lemma}\label{lem: Obs (8)}
Let $A$ be a ${\rm{C}}^*$-algebra and let $\rho$ be a state on $A$. Define $L_{\rho}, D_{\rho}=L_{\rho}\cap L_{\rho}^*$ with respect to $\rho$. The set $\cS$ of sequences $(a_1,a_2,\ldots ) \in \ell_\infty(A)$
with 
$$\lim_{n\to \omega}  \| a_n b \| + \| b a_n \|=0$$ 
for all $b\in A$
contains  $c_\omega(A)$  and 
is contained in $D_{\rho} = L_{\rho}^*\cap L_{\rho}$.
\end{lemma}
\begin{proof}It is clear that $c_{\omega}(A)\subset \mathcal{S}$. We show that $\mathcal{S}\subset D_{\rho}$. 
Let $(a_1,a_2,\dots)\in D_{\rho}$, and let $d_{\rho}\colon A\to \cL(H_{\rho})$ be the GNS representation of $A$ on a Hilbert space $H_{\rho}$ with respect to $\rho$ such that $\rho(a)=\nai{d_{\rho}(a)\xi_{\rho}}{\xi_{\rho}}$, where $\xi_{\rho}\in H_{\rho}$ is the corresponding cyclic vector. Then for each $b\in A$, we have 
\eqa{
\lim_{n\to \omega}\|d_{\rho}(a_nb)\xi_{\rho}\|\le \lim_{n\to \omega}\|a_nb\|=0.
} 
Since $d_{\rho}(A)\xi_{\rho}$ is dense in $H_{\rho}$ and $(a_1,a_2,\dots)$ is bounded, it follows that $d_{\rho}(a_n)\to 0$ strongly along $\omega$. In particular, we have 
$$\lim_{n\to \omega}\rho(a_n^*a_n)=\lim_{n\to \omega}\|d_{\rho}(a_n)\xi_{\rho}\|^2=0.$$
This shows that $(a_1,a_2,\dots)\in L_{\rho}$. Similar arguments show that $(a_1,a_2,\dots)\in L_{\rho}^*$, whence $\mathcal{S}\subset D_{\rho}$. 
\end{proof}
Notice that $\cS = \pi_\omega^{-1}(\Ann(A,A_\omega))$.

It follows that  
\begin{equation}
\Ann(A,A_\omega) \subseteq  \pi_\omega(D_{\rho}) \subseteq  A_\omega\,. \label{eq:  Obs(8)-1}
\end{equation}
Moreover, it is easy to see that (use $c_{\omega}(A)\subset D_{\rho}$)
\begin{equation}
\cN ( \pi_\omega (D_{\rho})) =  \pi_\omega (\cN (D_{\rho})).\label{eq: Obs (8)-2}
\end{equation}
In the rest of this section we explicitly distinguish $A$ and $\Delta (A)$ in order to state the next result without ambiguity. We denote by $\pi_{D_{\rho}}$ the quotient map $\mathcal{N}(D_{\rho})\to \mathcal{N}(D_{\rho})/D_{\rho}$, where $D_{\rho}$ is defined in terms of $(A,\rho)$. Thus for example, $\Ann (A, A_\omega)$
and $A'\cap A_\omega$ stand for $\Ann( \pi_\omega(\Delta(A)), A_\omega)$
respectively $\pi_\omega(\Delta(A))'\cap A_\omega$.
\begin{prop}\label{prop: Obs (9)-1} Let $(A,\rho)$ be as in Lemma \ref{lem: Obs (8)}. If $(A,\rho)$ satisfies the condition $\Delta(A)\subset \mathcal{N}(D_{\rho})$, then it holds that 
$$A'\cap (A,\rho)_\omega := \pi_{D_{\rho}}(\Delta(A))'\cap (\cN (D_{\rho})/D_{\rho})$$
is a quotient ${\rm C}^{\ast}$-algebra of a ${\rm C}^{\ast}$-subalgebra of $F(A)\,$. 
\end{prop}
\begin{proof}
Let $E_{\rho}:=\pi_\omega(D_{\rho})$, and let 
$\pi_{E_{\rho}}\colon \pi_\omega(\cN(D_{\rho}))\to 
\pi_\omega(\cN (D_{\rho}))/\pi_\omega(D_{\rho})$ 
be the quotient map. Note that $D_{\rho}=\pi_{\omega}^{-1}(\pi_{\omega}(D_{\rho}))$. Indeed, it is clear that $D_{\rho}\subset \pi_{\omega}^{-1}(\pi_{\omega}(D_{\rho}))$. On the other hand, let $x\in \pi_{\omega}^{-1}(\pi_{\omega}(D_{\rho}))$. Then there exists $y\in D_{\rho}$ such that $\pi_{\omega}(x)-\pi_{\omega}(y)=0$, i.e., $x-y\in c_{\omega}(A)\subset D_{\rho}$. 
Therefore $x\in y+D_{\rho}=D_{\rho}$ and we obtain $D_{\rho}\supset \pi_{\omega}^{-1}(\pi_{\omega}(D_{\rho}))$. 

We next observe that there is a *-isomorphism $\Phi\colon \mathcal{N}(D_{\rho})/D_{\rho}\to \mathcal{N}(\pi_{\omega}(D_{\rho}))/\pi_{\omega}(D_{\rho}) $. Since $\text{Ker}(\pi_{\omega}|_{\mathcal{N}(D_{\rho})})=c_{\omega}(A)\subset D_{\rho}$, $\pi_{\omega}|_{\mathcal{N}(D_{\rho})}$ factors through $\overline{\pi}_{\omega}\colon \mathcal{N}(D_{\rho})/D_{\rho}\to \mathcal{N}(\pi_{\omega}(D_{\rho}))=\pi_{\omega}(\mathcal{N}(D_{\rho})$. 
\[
\xymatrix{
\mathcal{N}(D_{\rho})\ar[d]^{\pi_{D_{\rho}}}\ar[r]^{\pi_{\omega}} & \mathcal{N}(\pi_{\omega}(D_{\rho}))\ar[d]^{\pi_{E_{\rho}}} \\
\mathcal{N}(D_{\rho})/D_{\rho}\ar@{-->}[ur]^{\overline{\pi}_{\omega}}\ar[r]^{\Phi\ } &\ \mathcal{N}(\pi_{\omega}(D_{\rho}))/\pi_{\omega}(D_{\rho}) 
}
\]
Therefore we set $\Phi=\pi_{E_{\rho}}\circ \overline{\pi}_{\omega}$, which is clearly surjective. To see that $\Phi$ is faithful, suppose $x\in \mathcal{N}(D_{\rho})$ satisfies $\pi_{E_{\rho}}\circ \overline{\pi}_{\omega}(x+D_{\rho})=\pi_{E_{\rho}}(\pi_{\omega}(x))=0$. Then $\pi_{\omega}(x)\in \text{Ker}(\pi_{E_{\rho}})=\pi_{\omega}(D_{\rho})$. Therefore $x\in \pi_{\omega}^{-1}(\pi_{\omega}(D_{\rho}))=D_{\rho}$. Therefore $\Phi$ is injective, whence a *-isomorphism.    
Next we see that:\\
\textbf{Claim.}
\begin{itemize}
\item[(i)] $\Phi^{-1}\circ \pi_{E_{\rho}}\colon \pi_{\omega}(\mathcal{N}(D_{\rho}))\to \cN(D_{\rho})/D_{\rho}$ maps $\Delta(a)+c_{\omega}(A)=\pi_{\omega}(\Delta (a))\ (a\in A)$ to $\Delta(a)+D_{\rho}$.
\item[(ii)] $\Phi^{-1}\circ \pi_{E_{\rho}} (\pi_{\omega}(\Delta(A))'\cap \pi_{\omega}(\mathcal{N}(D_{\rho})))=\pi_{D_{\rho}}(\Delta(A))'\cap \mathcal{N}(D_{\rho})/D_{\rho}$. 
\item[(iii)] $\Ann(A,A_\omega)\subset E_{\rho}=\pi_\omega(D_{\rho})$,
and  $\Ann(A,A_\omega)$ is an ideal of 
$C:=\pi_\omega(\Delta(A))' \cap \pi_\omega(\cN (D_{\rho}))\,.$
\end{itemize}
For (i), we have $\pi_{E_{\rho}}(\Delta(a)+c_{\omega}(A))=\pi_{\omega}(\Delta(a))+\pi_{\omega}(D_{\rho})=\Phi(\Delta(a)+D_{\rho})$.\\
To show (ii), by Proposition \ref{prop: Obs (7)}, $E_{\rho}=\pi_{\omega}(D_{\rho})$ is a $\sigma$-ideal of $\pi_{\omega}(\mathcal{N}(D_{\rho}))$. Since by the assumption that $\Delta(A)\subset \mathcal{N}(D_{\rho})$, $\pi_{\omega}(\Delta(A))$ is a separable C$^{\ast}$-subalgebra of $\pi_{\omega}(\mathcal{N}(D_{\rho}))$. Therefore by Theorem \ref{thm: pi_I is surjective}, 
$\pi_{E_{\rho}}$ maps $\pi_{\omega}(\Delta(A))'\cap \pi_{\omega}(\mathcal{N}(D_{\rho}))$ onto $\pi_{E_{\rho}}(\pi_{\omega}(\Delta(A)))'\cap \pi_{E_{\rho}}(\pi_{\omega}(\mathcal{N}(D_{\rho})))$. 
Since $\pi_{E_{\rho}}(\pi_{\omega}(\Delta(A)))=\Phi(\pi_{D_{\rho}}(\Delta(A)))$ by (i), we see that 
$$\Phi(\pi_{D_{\rho}}(\Delta(A))'\cap \mathcal{N}(D_{\rho})/D_{\rho})=\pi_{E_{\rho}}(\pi_{\omega}(\Delta(A)))'\cap \pi_{E_{\rho}}(\pi_{\omega}(\mathcal{N}(D_{\rho}))).$$
This proves (ii).\\
We show (iii). The first statement is proved in (\ref{eq:  Obs(8)-1}). To see that $\Ann(A,A_{\omega})$ is an ideal of $C=\pi_{\omega}(\Delta(A))'\cap \pi_{\omega}(\mathcal{N}(D_{\rho}))$, let $x=\pi_{\omega}(x_1,x_2,\dots)\in \Ann(A,A_{\omega})$ and let $y=\pi_{\omega}(y_1,y_2,\dots)\in C$. Then for every $a\in A$, 
$$yx\cdot \pi_{\omega}(\Delta(a))=0,\ \ xy\cdot \pi_{\omega}(\Delta(a))=x\pi_{\omega}(\Delta(a))y=0.$$
Thus $xy,yx\in \Ann(A,A_{\omega})$ and the claim is proved.
      
Now, since $\Ann(A,A_\omega)\subset E_{\rho}\,,$ by (\ref{eq:  Obs(8)-1}), the *-homomorphism $\Psi:=\Phi^{-1}\circ \pi_{E_{\rho}} | C\to A'\cap (A,\rho)_{\omega}=\pi_{D_{\rho}}(\Delta(A))'\cap \mathcal{N}(D_{\rho})/D_{\rho}$ factorizes through  $\overline{\Psi}\colon C/\Ann(A,A_\omega)\to A'\cap (A,\rho)_{\omega}\,.$ Then $B:= C / \Ann(A,A_\omega)$ is a \cst-subalgebra of $F(A)= (A'\cap A_\omega)/\Ann (A,A_\omega)\,,$ which is mapped by $\overline{\Psi}$ onto $A'\cap (A,\rho)_{\omega}\,.$ This finishes the proof. 
\end{proof}
Although it is not necessary to have a detailed discussion, we consider when the condition $\Delta(A)\subset \mathcal{N}(D_{\rho})$ is satisfied. 
\begin{prop}\label{prop: Obs (10)-1}
Let $(A,\rho)$ be as in Lemma \ref{lem: Obs (8)}. Then the following 6 conditions are all equivalent.
\begin{itemize}
\item[(i)]   $(a_1b,a_2b,\ldots ) \in L_{\rho}$ 
for all $b\in A$ and $(a_1,a_2,\ldots ) \in L_{\rho}$.

\item[(ii)] $\Delta(A)$  is contained in $\cN(D_{\rho})$.
\item[(iii)]  $\lim_{n\to\omega} \overline{\rho}(b^*a_n^*a_n b) =0$ 
for all $b\in A$ and $(a_1,a_2,\ldots)\in \ell_\infty(A^{**})$
with $\lim_{n\to\omega} \overline{\rho}(a_n^*a_n) =0$, 
where $\overline{\rho}$ denotes the unique extension to
$A^{**}$ of the state $\rho$ on $A$ to a normal state on $A^{**}$.

\item[(iv)] $\lim_{n\to \omega} \overline{\rho}(b^*a_n^*a_n b) =0$ 
for all $b\in A^{**}$ and 
$(a_1,a_2,\ldots)\in \ell_\infty(A^{**})$
with $\lim_{n\to \omega} \overline{\rho}(a_n^*a_n) =0$.

\item[(v)]  the support projection $p$  of $\rho$ 
(and of $\overline{\rho}$) in the second
conjugate $A^{**}$  of $A$ is in the center of  $A^{**}$.
\end{itemize}
\end{prop}

\begin{proof}
Since $L_{\rho}^*$ is a right closed ideal of $\ell_{\infty}(A)$, it is easy to see (i)$\Leftrightarrow $(ii). 
(i)$\Rightarrow$(iii):\,  
Let $(a_1,a_2,\ldots)\in \ell_\infty(A^{**})$
with 
$\lim_\omega \overline{\rho}(a_n^*a_n) =0$ and $b\in A$.

Consider a normal unital *-representation $d\colon A^{**}\to \cL(H)$
of $A^{**}$ on a Hilbert space space $H$, such that each 
normal state of $A^{**}$ is a vector state for a vector in $H$.
Let $\xi \in H$ be a unit vector such that  $\overline{\rho}(a)= \langle d(a)\xi,\xi \rangle$ holds
for all $a\in A^{**}$.   By Kaplansky density Theorem, for each $n\in \mathbb
N
$, there are $c_n\in A\subset A^{**}$ with $\| c_n \| \leq \| a_n\|$
and $\| d(c_n-a_n) \xi\| + \| d(c_n-a_n) d(b)\xi\|  \leq 2^{-n}$.
Then $(c_1,c_2,\ldots)\in \ell_\infty(A)$, and we have
$$\rho(c_n^*c_n)^{1/2}= \| d(c_n)\xi \|  \leq  \|d(a_n) \xi\| + 2^{-n}
= \overline{\rho}(a_n^*a_n)^{1/2}+ 2^{-n},\ \ \ n\in \mathbb{N}.$$
In particular  $(c_1,c_2,\ldots)\in L_{\rho}$ holds. Since  
$$\overline{\rho} (b^*(a_n^*a_n)b)^{1/2} = \| d(a_n)d(b)\xi\| 
\leq  \| d(c_n)d(b)\xi\| + 2^{-n}= \rho(b^*c_n^*c_nb)^{1/2}+ 2^{-n}\,,$$
it follows that $\lim_{n\to \omega} \overline{\rho} (b^*(a_n^*a_n)b)=0$.

\smallskip

(iii)$\Rightarrow$(iv):\,
Let $(a_1,a_2,\ldots)\in \ell_\infty(A^{**})$
with $\lim_{n\to \omega} \overline{\rho}(a_n^*a_n) =0$
and $b\in A^{**}$.
Then 
$$\gamma:=\lim_{n\to \omega} \overline{\rho}(b^*a_n^*a_nb)^{\frac{1}{2}}$$
is a well defined real number with 
$0\leq \gamma \leq \|b \| \cdot \sup_n \| a_n \|$.

Let $\varepsilon>0$.
We consider $A^{**}$ as a von Neumann-algebra in its 
universal strongly continuous representation
$d\colon A^{**}\to \cL(H)$.
Then the normal state $\overline{\rho}$
on
$A^{**}$ is a vector state 
$\overline{\rho}(a)= \langle d(a)\xi,\xi \rangle$
for $a\in A^{**}$ with $\xi \in H,\ \|\xi\|=1$.

Let $\delta := \varepsilon /  (1+ \sup_n \| a_n \|)>0$.
 By Kaplansky density Theorem, there exists $c\in A$ with
 $\| c \| \leq \|b\|$ and 
 $ \| d(b-c)\xi \| < \delta $.
 Then  $\lim_{n\to \omega} \overline{\rho}(c^*a_n^*a_nc) =0$
 by (iii), and  
 $$ 
 \overline{\rho}(b^*a_n^*a_nb)^{1/2}=\| d(a_n) d(b)x\|
 \leq \| d(a_n) d(c)x \| + \delta \cdot \| a_n \| <  
 \overline{\rho}(c^*a_n^*a_nc)^{1/2} +  \varepsilon\,.
 $$
Therefore it follows that
$$ \gamma=\lim_{n\to \omega} \overline{\rho}(b^*a_n^*a_nb)^{1/2}
\leq \varepsilon\,.$$
Since $\varepsilon>0$ is arbitrary, $\gamma=0$ holds, which proves (iv). 
\smallskip

(iv)$\Rightarrow$(v):\,  We show the contrapositive. Suppose that the support projection
$p\in A^{**}$ of the normal state  $\overline{\rho}$ is not in the center
of $A^{**}$.  Then the projections $p,1-p$ are not centrally orthogonal in $A^{**}$, so that there exists a non-zero partial isometry $u\in A^{**}$ with $u^*u\leq p$ and $uu^*\leq 1-p$. 
In particular $\overline{\rho}(uu^*)=0$
and $\overline{\rho}(u^*u)=: \gamma >0$.
Then $a_n:= u^*\ (n\in \mathbb{N})$ and $b:= u$ satisfy  $\overline{\rho}(a_n^*a_n)=0$ 
and $\overline{\rho}(b^*a_n^*a_nb)=\gamma$ for all $n\in \N$,
which contradicts property (iv).

(v)$\Rightarrow$(i):\, Assume (v). 
Consider the von Neumann algebra $M:= A^{**}p$ and a normal faithful state $\mu:=\overline{\rho}|_M$ on $M$ uniquely determined by the condition $\mu(ap)=\rho (a)\ (a\in A)$, where $p$
is the support projection of the normal state $\overline{\rho}$
on $A^{**}$ that extends $\rho$. 
We first show that $\Delta(M)\subset \mathcal{N}(D_{\mu})$. This is well-known in von Neumann algebra theory, but we include a proof for completeness. 
Let $d=d_{\mu}\colon M\to \cL(H)$ be the GNS representation of $M$ with respect to $\mu$ with the corresponding cyclic vector $\xi\in H$. Then $\xi$ is cyclic and separating for $d(M)=d(M)''$. 
Now let $(a_1,a_2,\ldots)\in \ell_\infty(M)$ with
$\lim_\omega  \mu(a_n^*a_n) =0$, and
$b\in M$. 
Define $\gamma \in [0, \|b\| \sup_n \| a_n\|]$
by $\gamma:= \lim_{n\to \omega}  \mu(b^*a_n^*a_nb)^{\frac{1}{2}}$.
Let $\varepsilon >0$ and 
$\delta:= \varepsilon /(1+ \sup_n \| a_n \|)$.
We are going to show that $ \gamma \leq \varepsilon$.
Recall that since $\xi$ is separating for $d(M)'$, $d(M)'\xi$ is dense in $H$. Therefore there exists $T\in d(M)'$ with $ \| T\xi - d(b)\xi \| <\delta$.
Hence, for all $a\in M$, 
$$\mu (b^*a^*ab)^{1/2} = 
\| d(a)d(b)\xi \| \leq \|d(a)T\xi\| + \delta  \cdot \| a\|$$ 
and
$$ \|d(a)T\xi\| = \| Td(a)\xi \|  \leq  \| T \| \mu (a^*a)^{1/2}\,.$$
Thus, for all $n\in \N$, 
$$\mu(b^*a_n^*a_nb)^{1/2} \leq \delta \cdot \sup_n \| a_n\|
+  \|T\| \mu(a_n^*a_n)^{1/2}\,.$$
Since $\lim_{n\to \omega} \mu(a_n^*a_n)^{1/2}=0$, it follows that 
$$\gamma=\lim_{n\to \omega} \mu (b^*a_n^*a_nb)^{1/2}\leq 
\delta \cdot \sup_n \| a_n \| < \varepsilon\,,
$$
which proves the claim. Therefore, we obtain $\lim_{n\to \omega}\mu(b^*a_n^*a_nb)=0$, so that $b\in \mathcal{N}(D_{\mu})$ by (i)$\Leftrightarrow $(ii).

Now suppose that $(a_1,a_2,\dots)\in L_{\rho}$ and $b\in A$ are given. Then set $\tilde{a}_n:=a_np, \tilde{b}:=bp\in M$. Then $\lim_{n\to \omega}\mu(\tilde{a}_n^*\tilde{a}_n)=\lim_{n\to \omega}\rho(a_n^*a_n)=0$, whence by the above argument, we have
$$\lim_{n\to \omega}\rho(b^*a_n^*a_nb)=\lim_{n\to \omega}\mu(\tilde{b}^*\tilde{a}_n^*\tilde{a}_n\tilde{b})=0\,.$$
This proves (i).   
\end{proof}
\subsection{Proof of Theorem \ref{thm: main2}}
We are now ready to prove Theorem \ref{thm: main2}. 
\begin{prop}\label{lem: Obs (12)-1}
Let $M$ be a von Neumann algebra with separable predual, $A$ be a weakly dense separable {\rm{C}}$^*$-subalgebra of $M$, and let $\rho$ be a normal faithful state on $M$.  Set $\mu:=\rho|_A$. Define $D_{\rho}:=L_{\rho}\cap L_{\rho}^*,\ L_{\rho}:=\{(x_1,x_2,\dots)\in \ell_{\infty}(M);\lim_{n\to \omega}\rho(x_n^*x_n)=0\}$ and define $L_{\mu},\ D_{\mu}=L_{\mu}\cap L_{\mu}^*\subset \ell_{\infty}(A)$, analogously. 
Then the following hold:
\begin{itemize}
\item[(i)] $\Delta(M)\subset \mathcal{N}(D_{\rho})$ and $\Delta(A)\subset \mathcal{N}(D_{\mu})$. 
\item[(ii)] There exists a *-isomorphism from the {\rm{C}}$^*$-to-{\rm{W}}$^*$ ultrapower $(A,\mu)_{\omega}$ onto the Ocneanu ultrapower $M^{\omega}$ which maps $A'\cap (A,\mu)_{\omega}=\pi_{D_{\mu}}(\Delta(A))'\cap \mathcal{N}(D_{\mu})/D_{\mu}$ onto $M'\cap M^{\omega}$. 
\end{itemize}
\end{prop}
\begin{proof}
(i): 
By (the proof of) (v)$\Rightarrow $(i) in Proposition \ref{prop: Obs (10)-1}, $\Delta(M)\subset \mathcal{N}(D_{\rho})$ holds. Alternatively, one can use the fact that the norm $\|\cdot \|_{\rho}^{\sharp}$ defines a *-strong topology on the unit ball of $M$ and the separate *-strong continuity of the operator product $(x,y)\mapsto xy$.  
Then $\Delta(A)\subset \mathcal{N}(D_{\rho})\cap \ell_{\infty}(A)=\mathcal{N}(D_{\mu})$.\\
\smallskip
(ii)  By (i), $A'\cap (A,\mu)_{\omega}$ is well-defined. Also, it follows that by (i) and Proposition \ref{prop: Obs (9)-1}, there exists a C$^*$-subalgebra $B$ of $F(A)$ and a closed ideal $J\triangleleft B$ such that $B/J$ is *-isomorphic to $A'\cap (A,\mu)_{\omega}$.  
We have seen in Proposition \ref{prop: connection between C* to W* and Ocneanu}, that the embedding  $A\subset M$ defines
a natural embedding of $\ell_\infty(A)$
into $\ell_\infty(M)$ with the property that
this embedding defines  an isomorphism
from $(A,\mu)_\omega$ onto $M^\omega$. Moreover, it is straightforward to see that this isomorphism maps $a\in A\subset M$ to 
$\pi_{D_{\rho}} (\Delta(a)) \in \cN( D_{\rho})/D_{\rho}$. That is: 
$$\mathcal{N}(D_{\mu})/D_{\mu}\supseteq \pi_{D_{\mu}}(\Delta(A))\ni \pi_{D_{\mu}}(\Delta(a))\leftrightarrow \pi_{D_{\rho}}(\Delta(a))\in \pi_{D_{\rho}}(\Delta(M))\subseteq  M^{\omega}=\mathcal{N}(D_{\rho})/D_{\rho}\,.$$

The relative commutant $A'\cap (A,\mu)_\omega$
maps in this way into $M'\cap M^\omega$.
Since the \cst-algebra 
$A$ is weakly dense in $M\subset M^{\omega}$,
it  follows that 
$$A'\cap (A,\rho|A)_\omega \cong  A'\cap M^{\omega}=(A'')'\cap M^{\omega}=M'\cap M^{\omega}\,.$$    
\end{proof}
We are now ready to prove Theorem \ref{thm: main2}. 
\begin{proof}[Proof of Theorem \ref{thm: main2}]
Choose a normal faithful state $\rho$ on $M$ and let $\mu:=\rho|_A$. Then by Lemma \ref{lem: Obs (12)-1} (i), $\Delta(A)\subset \mathcal{N}(D_{\mu})$. Thus by Proposition \ref{prop: Obs (9)-1} and Proposition \ref{lem: Obs (12)-1} (ii), we are done.  
\end{proof}

\section*{Appendix: Ozawa's proof of the non-commutativity of $A'\cap A_{\omega}$}
In this appendix, we give Ozawa's proof, based on Kishimoto-Ozawa-Sakai Theorem \cite{KOS03} that if $A$ is a  unital, separable and non-type I C$^{\ast}$-algebra, then the central sequence algebra $F(A)=A'\cap A_{\omega}$ is non-commutative. We thank him for allowing us to include it. We also thank the referee for the suggestion of adding the proof.   
Recall that since $A$ is separable, the automorphism group ${\rm{Aut}}(A)$ of $A$ is a Polish group with respect to the topology of pointwise norm-convergence. We denote by ${\rm{Inn}}(A)$ (resp. $\overline{\rm{Inn}}(A)$) the subgroup of all inner  (resp. approximately inner) automorphisms of $A$. We say that an automorphism $\alpha\in {\rm{Aut}}(A)$ is {\it centrally trivial}, if for every central sequence $(a_n)_{n=1}^{\infty}$ in $A$, the sequence $(\alpha(a_n)-a_n)_{n=1}^{\infty}$ tends to 0 in norm as $n\to \infty$.  The group of all centrally trivial automorphisms of $A$ is denoted by ${\rm{Ct}}(A)$.  
By the proof of Kishimoto-Ozawa-Sakai Theorem \cite{KOS03}, the following Theorem holds. 
\begin{thm}[Kishimoto-Ozawa-Sakai]\label{thm: KOS transitivity}
Let $A$ be a unital separable non-type {\rm{I}} ${\rm{C}}^*$-algbera $A$. Let $N\in \mathbb{N}$ and $(\pi_n)_{n=1}^N$ be a sequence of mutually non-equivalent irreducible representations of $A$ with the same kernel.  Then for each permutation $\sigma\in S_N$, there exists $\alpha\in \overline{\rm{Inn}}(A)$ such that $\pi_{n}\circ \alpha$ is unitarily equivalent to $\pi_{\sigma(n)}$ for every $1\le n\le N$.  
\end{thm}
\begin{rem}
Theorem \ref{thm: KOS transitivity} holds for $N=\infty$ and arbitrary permutation $\sigma\in S_{\infty}$, but the proof will be more difficult. For our purpose $N<\infty$ version suffices. 
The $N=\infty$ version of the theorem is used in \cite{AkemannWeaver04}. 
\end{rem}

The next Proposition is well-known or a folklore. The proof is nothing more than a copy of Connes' argument in \cite[Theorem2.2.1]{Connes75} on the characterization of McDuff factors of type II$_1$ by centrally trivial automorphisms. We nevertheless include the proof for completeness. 
\begin{prop}\label{prop: commutativity of AInn/Inn}
Let $A$ be a unital separable {\rm{C}}$^*$-algebra such that $A'\cap A_{\omega}$ is commutative.
Then the group $\overline{\rm{Inn}}(A)/{\rm{Inn}}(A)$ is commutative. 
\end{prop}
\begin{lemma}\label{lem: app inn cnt}
Let $A$ be a unital separable {\rm{C}}$^*$-algebra. Let $\varepsilon\colon {\rm{Aut}}(A)\to {\rm{Out}}(A)={\rm{Aut}}(A)/{\rm{Inn}}(A)$ be the canonical quotient. Then $\varepsilon(\overline{\rm{Inn}}(A))$ and $\varepsilon ({\rm{Ct}}(A))$ commute. 
\end{lemma}
\begin{proof}
Let $\theta\in {\rm{Ct}}(A)$ and $\alpha\in \overline{\rm{Inn}}(A)$. Then since $\theta$ is centrally trivial, for every $\varepsilon>0$, there exists an open neighborhood $\mathcal{V}$ of ${\rm{id}}_A$ in ${\rm{Aut}}(A)$ such that for every $u\in \mathcal{U}(A)$ we have the implication
\begin{equation}
{\rm{Ad}}(u)\in \mathcal{V}\Rightarrow \|\theta(u)-u\|<\varepsilon.
\end{equation}
Indeed, assume that this is not the case. Fix a neighborhood basis $\{\mathcal{U}_n\}_{n=1}^{\infty}$ of ${\rm{id}}_A$. Then there exists $\varepsilon>0$ such that for every $n\in \mathbb{N}$, there exists $u_n\in \mathcal{U}(A)$ with ${\rm{Ad}}(u_n)\in \mathcal{U}_n$ and $\|\theta(u_n)-u_n\|\ge \varepsilon$. This shows that $(u_n)_{n=1}^{\infty}$ is a central sequence with $\|\theta(u_n)-u_n\|\not\to 0$, a contradiction. 
Therefore we may find a decreasing sequence $(\mathcal{V}_n)_{n=1}^{\infty}$ of neighborhoods of ${\rm{id}}_A$ with $\bigcap_{n=1}^{\infty}\mathcal{V}_n=\{\text{id}_A\}$, such that for every $n\in \mathbb{N}$ and $u\in \mathcal{U}(A)$, we have 
\begin{equation}
{\rm{Ad}}(u)\in \mathcal{V}_n\Rightarrow \|\theta(u)-u\|<\frac{1}{2^n}.\label{eq: V_n condition}
\end{equation}
Choose a decreasing sequence $(\mathcal{W}_n)_{n=1}^{\infty}$ of neighborhoods of $\alpha$ in ${\rm{Aut}}(A)$ such that 
$\mathcal{W}_n\mathcal{W}_n^{-1}\subset \mathcal{V}_n\ (n\in \mathbb{N})$. Since $\alpha$ is approximately inner, for every $n\in \mathbb{N}$, there exists $u_n\in \mathcal{U}(A)$ such that ${\rm{Ad}}(u_n)\in \mathcal{W}_n$ holds. Then $\displaystyle \alpha=\lim_{n\to \infty}{\rm{Ad}}(u_n)$ and $\displaystyle \theta\circ \alpha \circ \theta^{-1}=\lim_{n\to \infty}{\rm{Ad}}(\theta(u_n))$. Set $v_n:=u_{n+1}u_n^*\in \mathcal{U}(A)\ (n\in \mathbb{N})$. Then for each $n\in \mathbb{N}$, we have 
$${\rm{Ad}}(v_n)\in \mathcal{W}_{n+1}\mathcal{W}_n^{-1}\subset \mathcal{W}_n\mathcal{W}_n^{-1}\subset \mathcal{V}_n,$$
so that $\|\theta(v_n)-v_n\|<2^{-n}$ by (\ref{eq: V_n condition}). Therefore for every $n\in \mathbb{N}$, it holds that 
\eqa{
\|u_{n+1}^*\theta(u_{n+1})-u_n^*\theta(u_n)\|&=\|\theta(u_{n+1})-u_{n+1}u_n^*\theta(u_n)\|\\
&=\|(\theta(v_n)-v_n)\theta(u_n)\|\\
&=\|\theta(v_n)-v_n\|<2^{-n}.
} 
This shows that $(u_n^*\theta(u_n))_{n=1}^{\infty}$ is a Cauchy sequence, so that $\displaystyle w=\lim_{n\to \infty}u_n^*\theta(u_n)\in \mathcal{U}(A)$ exists, and 
\[\alpha^{-1}\circ \theta \circ \alpha \circ \theta^{-1}={\rm{Ad}}(w)\in {\rm{Inn}}(A).\]
This shows that $\varepsilon(\theta)$ and $\varepsilon(\alpha)$ commute. 
\end{proof}
\begin{proof}[Proof of Proposition \ref{prop: commutativity of AInn/Inn}]
Assume that $A'\cap A_{\omega}$ is abelian. Let $\theta\in \overline{\rm{Inn}}(A)$. We show that $\theta$ is centrally trivial. Fix a dense subset $\{a_n\}_{n=1}^{\infty}$ of the closed unit ball of $A$. We first show:\\ \\
\textbf{Claim.} For every $\varepsilon>0$, there exist $\delta>0, n\in \mathbb{N}$ and $b_1,\dots, b_n\in A$ with $\|b_j\|\le 1\ (1\le j\le n)$ such that  if $x,y\in A$, then 
\begin{equation}
\|x\|\le 1,\ \|y\|\le 1,\ \|[x,b_j]\|<\delta,\ \|[y,b_j]\|<\delta\ \ (1\le j\le n)\Rightarrow \|[x,y]\|<\varepsilon. 
\end{equation} 
Indeed, assume that this is not the case. Let $\{b_j\}_{j=1}^{\infty}$ be a countable dense subset of the unit ball of $A$. Then there exists $\varepsilon>0$ such that for every $n\in \mathbb{N}$, there exist $x_n,y_n\in A$ with $\|x_n\|\le 1,\ \|y_n\|\le 1$ such that 
\[\|[x_n,b_j]\|<\frac{1}{n},\ \|[y_n,b_j]\|<\frac{1}{n},\ (1\le j\le n),\ \ \ \text{and\ \ } \|[x_n,y_n]\|\ge \varepsilon.\]
Then $x:=\pi_{\omega}((x_n)_{n=1}^{\infty}), y:=\pi_{\omega}((y_n)_{n=1}^{\infty})\in A'\cap A_{\omega}$ do not commute, a contradiction. Let $\varepsilon>0$ and choose $\delta>0$ and $b_1,\dots, b_n$ as in the Claim. Define an open neighborhood $\mathcal{V}$ of ${\rm{id}}_A$ in ${\rm{Aut}}(A)$ by 
\[\mathcal{V}:=\{\alpha \in {\rm{Aut}}(A); \|\alpha(b_j)-b_j\|<\delta\ \  \ (1\le j\le n)\}.\]
We observe that if $x\in A$ satisfies $\|x\|\le 1$ and $\|[x,b_j]\|<\delta\ (1\le j\le n)$, then for every 
$\alpha={\rm{Ad}}(u)\in \mathcal{V}\cap {\rm{Inn}}(A)$, we have $\|[u,b_j]\|=\|\alpha(b_j)-b_j\|<\delta$, so that by Claim, 
\[\|[x,u]\|=\|\alpha(x)-x\|<\varepsilon.\]
Since $\mathcal{V}$ is open, $\overline{\rm{Inn}}(A)\cap \mathcal{V}\subset \overline{{\rm{Inn}}(A)\cap \mathcal{V}}$ holds, so that we also have 
\begin{equation}
\|x\|\le 1,\ \  \|[x,b_j]\|<\delta\ (1\le j\le n)\Rightarrow \|\alpha(x)-x\|\le \varepsilon\ \ \ (\alpha \in \overline{{\rm{Inn}}}(A)\cap \mathcal{V}).\label{eq: [x,bj]}
\end{equation} 
Since $\theta\in \overline{{\rm{Inn}}}(A)$, we may find $w\in \mathcal{U}(A)$ such that $\alpha=\theta \circ {\rm{Ad}}(w)^{-1}\in \mathcal{V}\cap \overline{{\rm{Inn}}}(A)$ holds. Since $\{a_n\}_{n=1}^{\infty}$ is dense in the unit ball of $A$, we may find $0<\delta'\le \delta$ and $b_{n+1},\dots, b_{m}$ with $\|b_j\|\le 1\ (n+1\le j\le m)$ such  that for $y\in A$, 
\[\|y\|\le 1,\ \ \|[y,b_j]\|<\delta'\ \ (n+1\le j\le m)\Rightarrow \|[w,y]\|<\varepsilon.\]
Then if $x\in A$ satisfies $\|x\|\le 1$ and $\|[x,b_j]\|<\delta'\ (1\le j\le m)$, we have 
\eqa{
\|\theta(x)-x\|&=\|\alpha(wxw^*)-x\|\\
&\le \|\alpha (wxw^*-x)\|+\|\alpha(x)-x\|\\
&=\|[x,w]\|+\|\alpha(x)-x\|<2\varepsilon.
}
Since $\varepsilon>0$ is arbitrary, this shows that $\theta\in {\rm{Ct}}(A)$, whence $\overline{\rm{Inn}}(A)\subset {\rm{Ct}}(A)$ holds. Therefore $\overline{\rm{Inn}}(A)/{\rm{Inn}}(A)$ is commutative by Lemma \ref{lem: app inn cnt}.  
\end{proof}
\begin{proof}[Proof of the non-commutativity of $A'\cap A_{\omega}$] 
Fix an integer $N\ge 3$. 
By Glimm's Theorem \cite{Glimm61}, there exists a sequence $(\pi_n)_{n=1}^N$ of mutually non-equivalent irreducible representations of $A$ with the same kernel (in fact, there exist uncountably many such representations). Let $\sigma\in S_N$. Then by Theorem \ref{thm: KOS transitivity}, there exists $\alpha_{\sigma}\in \overline{\rm{Inn}}(A)$ such that $\pi_n\circ \alpha_{\sigma}$ is unitarily equivalent to $\pi_{\sigma^{-1}(n)}$ for all $1\le n\le N$. Let $\widehat{A}$ be the space of all unitary equivalence classes of irreducible representations of $A$. Then since inner automorphisms preserve the unitary equivalence classes of irreducible representations, we have a surjective homomorphism $\beta$
from a subgroup $G_N$ of the quotient group $\overline{\rm{Inn}}(A)/{\rm{Inn}}(A)$ onto $S_N$ given by 
$G_N\ni [\alpha_{\sigma}]\mapsto \sigma\in S_{N}$ with $[\pi_n\circ \alpha_{\sigma}]=[\pi_{\sigma^{-1}(n)}]\ (1\le n\le N)$, where $[\pi]$ is the class of $\pi$ in $\widehat{A}$. This shows that  $\overline{\rm{Inn}}(A)/{\rm{Inn}}(A)$ is non-commutative, whence $A'\cap A_{\omega}$ is non-commutative by Proposition \ref{prop: commutativity of AInn/Inn}. 
\end{proof}
\section*{Acknowledgements}
The authors would like to thank Professor Mikael R\o rdam and Dr. James Gabe for useful comments and for careful reading on the first (resp. the second) draft of the paper, Professor Narutaka Ozawa for informing us of his proof of the noncommutativity of $A'\cap A_{\omega}$ using \cite{KOS03}. We also thank the anonymous referee for various useful suggestions which improved the presentation of the paper. The first named author was supported by the Danish National Research Foundation through the Centre for Symmetry and Deformation (DNRF92) while he was working as a postdoc researcher on this article at University of Copenhagen. 
He is supported by JSPS KAKENHI 16K17608 for the last part of the project. 
\providecommand{\bysame}{\leavevmode\hbox
to3em{\hrulefill}\thinspace}

%

\begin{thebibliography}{99}
\bibitem[AP79]{AkemannPedersen} C. A. Akemann and Gert K. Pedersen, 
{\em Central sequences and inner derivations of separable C$^*$-algebras}, Amer.~Journ. Math. \textbf{101} (1979), 1047--1061.
\bibitem[AH14]{AndoHaagerup} 
H. ~Ando and U.~Haagerup, 
{\em Ultraproducts of von Neumann algebras}, 
 Journal of Functional Analysis \textbf{266} (2014), 6842--6913.
%
\bibitem[Co75]{Connes75} A.~Connes, Outer conjugacy classes of automorphisms of factors, 
{\it Ann. Sci. Ecole Norm. Sup.} (4) \textbf{8} (1975), 383--419. 
%
\bibitem[AW04]{AkemannWeaver04} C. Akemann and N. Weaver, Consistency of a counterexample to Naimark's problem, {\it  Proc. Natl. Acad. Sci. USA} \textbf{101} (2004), 7522--7525. 
%
\bibitem[EW76]{ElliottWoods76} G. A. Elliott; E. J. Woods, 
{\em The equivalence of various definitions for a properly infinite von Neumann algebra to be approximately finite dimensional},   Proc. Amer. Math. Soc. \textbf{60} (1976), 175--178.
%
\bibitem[Far09]{Farah09} I.~Farah, 
{\em The relative commutant of separable C*-algebras of real rank zero}, 
J.~Funct.~Anal, \textbf{256} (2009), 3841--3846.

\bibitem[FPS10]{FarahPhillipsSteprans} I.~Farah, N.~Christopher~Phillips~and J.~Steprans, 
{\em The commutant of $L(H)$ in its ultrapower may or may not be trivial}, 
Math. Ann, \textbf{347} (2010), 839--857.
\bibitem[Gli61]{Glimm61} J.~Glimm, {\em Type I \cst--algebras},
Ann.~of Math.~\textbf{73} (1961), 572 -- 612.
%
%
%
%
%
%
\bibitem[Kir94]{Kirchberg94} 
E.~Kirchberg, 
{\em Commutants of unitaries in UHF algebras and functorial properties 
of exactness}, 
J. reine. angew. Math. \textbf{452} (1994), 39 -- 77. 
%
%
\bibitem[Kir95]{Kirchberg95} 
E.~Kirchberg, 
{\em On restricted perturbations in inverse images and a description
of normalizer algebras in C$^*$-algebras},  
J. Funct. Analysis \textbf{129} (1995), 1 -- 34.
%
%
%
%


\bibitem[Kir04]{Kir.AbelProc} 
\bysame,
{\em Central sequences in C$^*$-algebras and strongly purely infinite algebras}, 
p.175--231, in 
{\em Operator Algebras: The Abel Symposium 2004},
Proceedings of the First Abel Symposium, Oslo, September 3-5, 2004,
Bratteli, Ola;  Neshveyev, Sergey;  Skau, Christian (Eds.),
Springer, Berlin Heidelberg  2006, X, 279 p.
%
\bibitem[KOS03]{KOS03} A. Kishimoto, N. Ozawa and S. Sakai, 
{\em Homogeneity of the pure state space of a separable C*-algebra}, \\
Canad.~Math.~Bull.,~{\bf 46} (2003), 365--372.
%
\bibitem[KR13]{KirRor.Crelle2013} \bysame,
M. R\o rdam,
{\em Central sequence C*-algebras and tensorial absorption of the Jiang-Su algebra}
J.~reine angew.~Math.~{\bf 0} (2013), 0--0. (Crelle Journal),
Published online  13-03-09
%
\bibitem[KR14]{KirchbergRordam14}
\bysame \bysame, {\em When central sequence algebras have characters}, to appear in {\it International J. Math.} (arXiv:1409.1395).
%
\bibitem[Mar75]{Marechal75}
O.~Mar{\'e}chal,
{\em Une remarque sur un th{\'e}or{\`e}me de Glimm},
Bull.~Sci.\ Math. (2) \textbf{99} (1975), no. 1, 41-- 44.
%

\bibitem[Oc85]{Ocneanu85} A.~Ocneanu, 
{\em Actions of discrete amenable groups on von Neumann algebras}, 
Lecture Notes in Mathematics \textbf{1138}, Springer (1985).

\bibitem[Sa11]{Sato11} Y.~Sato, {\em Discrete amenable group actions on von Neumann algebras and invariant nuclear ${\rm C}^\ast$-subalgebras}, arXiv:1104.4339.
\end{thebibliography}
\end{document}